\documentclass[preprint,12pt]{elsarticle}

\usepackage{amssymb}

\journal{Discrete Optimization}

\usepackage{graphicx}
\usepackage{amsfonts}
\usepackage{mathrsfs}
\usepackage{amsmath}
\usepackage{amssymb}
\usepackage{tikz} 
\usepackage{csquotes}
\usepackage{amsthm}
\usepackage{subcaption}

\newcommand{\R}{\mathbb{R}}
\newcommand{\N}{\mathbb{N}}
\newcommand{\Z}{\mathbb{Z}}

\newcommand{\proj}{\operatorname{proj}}

\newcommand{\cl}{\operatorname{cl}}
\newcommand{\conv}{\operatorname{conv}}
\newcommand{\cone}{\operatorname{cone}}

\newcommand{\rec}{\operatorname{rec}}
\newcommand{\supp}{\operatorname{supp}}

\newcommand{\lin}{\operatorname{lin}}
\newcommand{\st}{\text{s.t.}}
\newcommand{\eg}{\text{e.g.}}

\newcommand{\I}{\mathscr{I}}
\newcommand{\F}{\mathscr{F}}

\newcommand{\transp}{\mathsf T}

\newtheorem{theorem}{Theorem}
\newtheorem{definition}{Definition}
\newtheorem{lemma}{Lemma}
\newtheorem{proposition}{Proposition}
\newtheorem{observation}{Observation}
\newtheorem{corollary}{Corollary}

\newtheorem{example}{Example}

\makeatletter
\newtheorem*{rep@theorem}{\rep@title}
\newcommand{\newreptheorem}[2]{%
\newenvironment{rep#1}[1]{%
 \def\rep@title{#2 \ref{##1}}%
 \begin{rep@theorem}}%
 {\end{rep@theorem}}}
\makeatother

\newreptheorem{theorem}{Theorem}

\newcounter{claim} 
\newenvironment{claim}[1][]
{\refstepcounter{claim} \begin{trivlist} \item[] {\bf Claim~\theclaim.}\space#1 \itshape}
{\end{trivlist}}

\newenvironment{cpf}
{\begin{trivlist} \item[] {\em Proof of claim. }}
{$\hfill\diamond$ \end{trivlist}}

\begin{document}

\begin{frontmatter}



\title{The $k$-aggregation Closure for Covering Sets}


\author{Haoran Zhu}
\ead{hzhu94@wisc.edu}


\address{Department of Industrial and Systems Engineering, University of Wisconsin-Madison, Madison, WI, USA}

\begin{abstract}
In this paper, we will answer one of the questions proposed by Bodur, Del~Pia, Dey, Molinaro and Pokutta in 2017. Specifically, we show that the $k$-aggregation closure of a covering set is a polyhedron. The proof technique is based on an equivalent condition for the closure of any particular family of cutting-planes to be polyhedral, from the perspective of convex geometry. We believe that this technique can be applied to tackle other polyhedrality problems in the future and may be of independent interest.


\end{abstract}

%


\end{frontmatter}



\section{Introduction}

Cutting-plane technology has been one of the main pillars in the edifice of state-of-the-art mixed integer programming (MIP) solvers for a few decades, and numerous types of cutting-planes have been extensively studied in the literature, both theoretically and computationally (see, for example, \cite{conforti2010polyhedral,cornuejols2008valid,johnson2000progress,MR1922341,richard2010group}).

In general, a cutting-plane (or simply, a \emph{cut}) for a polyhedron $P$ is an inequality that separates some fractional point of $P$ from all integer points in $P$, and when added to the polyhedron itself, typically yields a stronger relaxation for its \emph{integer hull} $\conv(P \cap \Z^n)$, compared with its natural linear relaxation. Here $n$ is the dimension of the ambient space.
Normally speaking, given one specific cut generating method, there are infinitely many possible cutting-planes that can be obtained, e.g., Chv\'atal-Gomory cuts \cite{MR313080}, split cuts \cite{cook1990chvatal} etc. 
For a family of cutting-planes, the \emph{(cutting-plane) closure} is the intersection of all the half-spaces given by the cutting-planes from that family \cite{MR2969261}. 
Closure is of both theoretical and practical interest, since the polyhedrality of a closure essentially indicates that there exists only finitely many ``strongest'' cutting-planes from this family, and the integrality gap obtained from a closure can be used to measure the ``strength'' of this particular cutting-plane family.
We give its formal definition in the following:

\begin{definition}
\label{def: closure}
For a family of half-spaces $\{H^i\}_{i \in I}$ in $\mathbb{R}^n$, each of the half-space $H^i$ is given by an inequality $(\alpha^i)^\transp  x \leq \beta_i$, then the \textbf{closure} for that family of half-spaces is:
$$\mathscr{I}: = \bigcap_{i \in I} \ H^i = \bigcap_{i\in I} \{x \in \mathbb{R}^n \mid (\alpha^i)^\transp  x \leq \beta_i\}.$$
\end{definition}

For the ease of notation, we will normally use an upper index to denote vector, and a lower index to denote number, $\eg, \alpha^i \in \R^n$, $\beta_i \in \R$.

For any $\Omega \subseteq \R^{n+1}$, we use ``a family of cuts given by $\Omega$'' to refer to the family of cuts given by $\alpha^\transp x \leq \beta,$ for any $(\alpha, \beta) \in \Omega, \alpha \in \R^n, \beta \in \R$. 
We also denote the corresponding closure as:
\begin{align}
\label{def: closure2}
\mathscr{I}(\Omega): = \bigcap_{(\alpha, \beta) \in \Omega} \{x \in \R^n \mid \alpha^\transp x \leq \beta\}.
\end{align}
Clearly this closure is convex and closed, but whenever $\Omega$ is an infinite set, this closure is obtained by intersecting infinitely many half-spaces. One natural question then arises: When $\Omega$ is given by some specific cut-generating procedure, is $\mathscr{I}(\Omega)$ polyhedral? As an equivalent condition for $\mathscr{I}(\Omega)$ to be polyhedral, we have the following characterization theorem.

\begin{theorem}
\label{theo: equiv}
$\Omega$ is a set in $\R^{n+1}$ that contains $(0, \ldots, 0, 1)$, and
let $L = \lin(\cl \cone(\Omega))$. Then $\I(\Omega)$ is polyhedral iff $\cl \cone(\proj_{L^\perp} \Omega)$ has finitely many different extreme rays. 
\end{theorem}

Here for any set $S \subseteq \R^n, \lin(S): = S \cap (-S)$ denotes the \emph{linearity space} of $S$, which is the largest linear subspace contained in $S$. For a linear subspace $L$, we denote by $L^\perp$ the orthogonal complement of $L$. $\cone(S)$ denotes the \emph{conical hull} of $S$, which is defined as the set of all conical combination of finitely many points in $S: \cone(S) = \{\sum_{i=1}^t \alpha_i s^i \mid s_i \in S, \alpha_i \geq 0, i \in [t] \text{ for some } t \in \N\}$. We use $\cl(S)$ to denote the smallest closed set containing $S$, which is also called \emph{closure} in topology. To avoid confusion, we will only use $\cl(S)$ to refer to the topological closure, and whenever we say \emph{closure} in this paper, it refers to the cutting-plane closure in Definition~\ref{def: closure}, or equation~\eqref{def: closure2}. 

People have been looking at such polyhedrality for various closures in literature, see, $\eg$, \cite{MR2676765,MR2969261,dash2020generalized,MR3670262,MR4207341,vielma2007constructive}. In this paper, using the characteristical Theorem~\ref{theo: equiv}, we are going to add one more closure into the list: the $k$-aggregation closure of a covering set is polyhedral. 

\subsection{$k$-aggregation Closure}
A \emph{packing polyhedron} is of the form $\{x \in \R^n_+ \mid  Mx \leq d\}$ where $M \in \R^{m \times n}_+, d \in \R^m_+$ for some $m, n \in \N$. 
Analogously, a \emph{covering polyhedron} is of the form $\{x \in \R^n_+ \mid  Mx \leq d\}$ where $(M,d) \in (\R^{m \times n}_-, \R^{m}_-)$.
In a paper by Bodur et~al. \cite{bodur2017aggregation}, they also defined a \emph{covering set} to be the non-polyhedral counterpart of covering polyhedron: it is a set of the form $\{x \in \R^n_+ \mid M_i x \leq d_i \ \forall i \in I\}$ where each $(M_i, d_i) \in (\R^{ n}_-, \R_-)$ and $I$ is an arbitrary set (potentially infinite).
The non-polyhedral \emph{packing set} can be defined analogously.

The authors of \cite{bodur2017aggregation} also presented a new class of closure, namely the \emph{k-aggregation closure}, for some $k \in \N$. For a  polyhedron $Q = \{x \in \R^n_+ \mid Mx \le d\}$, its \emph{k-aggregation closure} is defined as:
\begin{equation}
\label{eq: defn_k_aggregation}
\mathcal{A}_k(Q): = \bigcap_{\lambda^1, \ldots, \lambda^k \in \R^{m}_+}\conv(\{x \in \N^n \mid (\lambda^j)^\transp M x \leq (\lambda^j)^\transp d \ \forall j \in [k]\}).
\end{equation}
This can easily be seen as a tighter relaxation than the classic Chvat\'al-Gomory closure for the integer hull of $Q$.
For non-polyhedral covering set however, since here the number of inequalities within $Q$ is infinite, then the multiplier $\lambda^j$ in \eqref{eq: defn_k_aggregation} will be an infinite-dimensional vector, which can cause ambiguity from the ill-defined inner product $(\lambda^j)^\transp  M$ of infinite-dimensional vectors, since the infinite series summation of product terms might diverge.
To deal with this infinite-dimensional issue, Del Pia et al. \cite{MR4207341} introduced a slightly different and more general definition for the $k$-aggregation closure, instead of using those multipliers:
for a convex set $S$ in $\R^n_+$, denote by 
\begin{equation}
\label{eq: Lambda_S}
\Lambda(S): = \{f \in \R^n \mid \sup\{f^\transp  x \mid x \in S\} < \infty \}.
\end{equation}
 Then the $k$-aggregation closure of $S$ is
\begin{equation}
\label{eq: defn_generalized_k_agg}
\tilde{\mathcal{A}_k}(S): = \bigcap_{f^1, \ldots, f^k \in \Lambda(S)} \conv(\{x \in \N^n \mid (f^j)^\transp x \leq \sup\{(f^j)^\transp s \mid s \in S\}, \ \forall j \in [k]\}).
\end{equation}
In other words, $\tilde{\mathcal{A}_k}(S)$ is defined as the intersection of all integer hulls of the polyhedra composed by $k$ valid inequalities of $S$.
It was observed in \cite{MR4207341} that for any packing polyhedron $P$ and $k \geq 1$, $\tilde{\mathcal{A}_k}(P)$ coincides with $\mathcal{A}_k(P)$. Actually follows from almost the same argument, this also holds for a covering polyhedron. So the $k$-aggregation closure defined in \eqref{eq: defn_generalized_k_agg} extends the original definition \eqref{eq: defn_k_aggregation} of \cite{bodur2017aggregation}. Therefore, throughout this paper, we will simply denote the $k$-aggregation closure in \eqref{eq: defn_generalized_k_agg} by $\mathcal{A}_k$.

Given that the $k$-aggregation closure in \eqref{eq: defn_k_aggregation} is obtained by intersecting infinitely many polyhedra, the polyhedrality of this closure is unclear. The authors of \cite{bodur2017aggregation} showed that for packing and covering polyhedron, if the matrix $M$ is fully dense, then the closure is still polyhedral. For the general case however, they proposed this question as an open problem.

In the paper of Del Pia et al. \cite{MR4207341}, by using the generalized definition in \eqref{eq: defn_generalized_k_agg}, they showed that the polyhedrality of the $k$-aggregation closure for the \textbf{packing set} can be automatically derived from the structural study of some discrete object called \emph{integer packing sets},
however such structural property cannot be carried over to the covering case.
An independent and almost simultaneous proof was given by Pashkovich et al. in \cite{pashkovich2021aggregation}, where they proved that the $k$-aggregation closure of both packing and covering polyhedron are still polyhedral.
In this paper, we are going to show the following main theorem: for any \textbf{covering set}, the corresponding $k$-aggregation closure as defined in \eqref{eq: defn_generalized_k_agg} is still polyhedral. Therefore, together with \cite{MR4207341}, we are able to augment the results of \cite{pashkovich2021aggregation} to the non-polyhedral case.

\begin{theorem}
\label{theo: aggregation_sign_pattern}
Let $Q$ be a covering set, then $\mathcal{A}_k(Q)$ is polyhedral, for any $k \in \N$. 
\end{theorem}

This paper is organized as follows. In Sect.~\ref{sec: 2} we show Theorem~\ref{theo: equiv}, and give the actual characterization for extreme rays of the $\cl \cone (\proj_{L^\perp} \Omega)$ in Theorem~\ref{theo: equiv}.
In Sect.~\ref{sec: 3}, we apply the proof technique we developed from Sect.~\ref{sec: 2} to prove the main Theorem~\ref{theo: aggregation_sign_pattern}. 

\paragraph{Notation}
For a positive integer $n$, we set $[n]: = \{1,\ldots, n\}$. For sets $S \subseteq \R^n$ and $I \subseteq [n]$, denote by $\proj_{I} S$ the orthogonal projection of $S$ onto the space of variables with indices in $I$. For a vector $x \in \R^n, x_{[t]}: = (x_1, \ldots, x_t)$ denotes the projection of $x$ onto the first $t$ components. 
For a linear subspace $L \subseteq \R^n$, we denote by $\proj_L x$ the orthogonal projection of vector $x$ onto $L$, and for any $X \subseteq \R^n, \proj_{L} X :=\{\proj_L x \mid x \in X\}$. $(X)_+ := \{\lambda x \mid \lambda \geq 0, x \in X\}$ denotes the cone that contains all non-negative multiplication of elements in $X$. 
For a closed convex set $S, \rec(S):=\{r \mid s + \lambda r \in S, \forall s \in S, \lambda \geq 0\}$ denotes the \emph{recession cone} of $S$. Given a closed convex cone $K$, we say $K$ is \emph{pointed} if $K \cap (-K) = \{0\}$.
When the dimension of the ambient space is clear, we denote by $e^j$ the $j$-th unit vector.

\section{Equivalent Condition for the Polyhedrality of Closure}
\label{sec: 2}

\subsection{Proof of Theorem~\ref{theo: equiv}}
We devote the first part of this section to the proof of Theorem~\ref{theo: equiv}. The following \emph{extended Farkas' Lemma} is not too well-known and we explicitly state it here.
\begin{theorem}[Extended Farkas' Lemma, Corollary 3.1.2 in \cite{goberna1998linear}]
\label{theo: extended_farkas_lemma}
The inequality $a^\transp x \leq b$ is a consequence of the non-empty system $\{a_t^\transp x \leq b_t \ \forall t \in T\}$ iff $(a,b) \in \cl \cone \{(a_t, b_t) \ \forall t \in T, (0, \ldots, 0, 1)\}$.
\end{theorem}

When $T$ is a finite set, this theorem reduces to the standard Farkas' Lemma.
From the definition of \emph{closure} in Definition~\ref{def: closure}, we can immediately obtain the following result.
\begin{proposition}
\label{prop: valid_ineq_for_closure}
Given a set $\Omega \subseteq \R^{n+1}$ which contains $(0, \ldots, 0, 1)$, and $\I(\Omega) \neq \emptyset$. Then $\alpha^\transp x \leq \beta$ is a valid inequality to $\I(\Omega)$ iff $(\alpha, \beta) \in \cl \cone(\Omega)$.
\end{proposition}

\begin{proof}
Consider the linear system $\{x \in \R^n \mid \omega^\transp (x, -1) \leq 0\ \forall \omega \in \Omega\}$. Since $\I(\Omega)$ is the feasible region given by this linear system which is also non-empty, we know that linear system $\{x \in \R^n \mid \omega^\transp (x, -1) \leq 0\ \forall \omega \in \Omega\}$ is non-empty.
By the Extended Farkas' Lemma~\ref{theo: extended_farkas_lemma} and the assumption that $(0, \ldots, 0, 1) \in \Omega$, we know $\alpha^\transp x \leq \beta$ is valid to $\I(\Omega)$ iff, $(\alpha, \beta) \in \cl \cone (\Omega)$.
\end{proof}

The above proposition has the following obvious implication, and we omit its proof here.
\begin{corollary}
\label{coro: iff_polyhedral_cone}
Given a set $\Omega \subseteq \R^{n+1}$ which contains $(0, \ldots, 0, 1)$, and $\I(\Omega) \neq \emptyset$. Then $\I(\Omega)$ is a polyhedron iff $\cl \cone(\Omega)$ is a polyhedral cone.
\end{corollary}

From this above Corollary~\ref{coro: iff_polyhedral_cone}, we know that in order to show $\I(\Omega)$ is a polyhedron it suffices to show that $\cl \cone(\Omega)$ is a polyhedral cone. The next lemma is helpful for characterizing $\cl \cone(\Omega)$. Here $\oplus$ denotes the \emph{direct sum}. 

\begin{lemma}
\label{lem: easy_lemma}
For any $\Omega \subseteq \R^{n+1}$, let $L = \lin(\cl \cone(\Omega))$. Then 
$\cl \cone(\Omega) =\cl \cone(\proj_{L^\perp} \Omega) \oplus L$, where $\cl \cone(\proj_{L^\perp} \Omega)$ is a pointed, closed convex cone.
\end{lemma}

Recall that a cone is called \emph{pointed} if its linearity space is the origin.
In order to prove the above lemma, we will need the following result.

\begin{lemma}[fact 9 \cite{studeny1993convex}]
\label{lem: cone_direct_sum}
Given a non-empty closed convex cone $K$, $K \cap \lin(K)^\perp$ is a pointed cone and $K = (K \cap \lin(K)^\perp) \oplus \lin(K)$. 
\end{lemma}

Now we are going to prove Lemma~\ref{lem: easy_lemma}. 

\begin{proof}[Proof of Lemma~\ref{lem: easy_lemma}]
By Lemma~\ref{lem: cone_direct_sum}, it suffices to show: $\cl \cone(\proj_{L^\perp} \Omega) = \cl \cone(\Omega) \cap L^\perp$.
First, we want to show $\cl \cone(\proj_{L^\perp} \Omega) \subseteq \cl \cone(\Omega) \cap L^\perp$. The relation $\cl \cone(\proj_{L^\perp} \Omega) \subseteq \cl \cone(L^\perp) = L^\perp$ is obvious. Moreover, for any $\omega \in \Omega, \ \proj_{L^\perp} \omega = \omega - r,$ for some $r \in L$. Hence $\proj_{L^\perp} \Omega \subseteq \Omega + L \subseteq \cl \cone(\Omega) + \cl \cone(\Omega) = \cl \cone(\Omega)$. Therefore, $\cl \cone(\proj_{L^\perp} \Omega) \subseteq \cl \cone(\Omega)$, which completes the proof of this $\subseteq$ direction. 

Then, we want to show that $\cl \cone(\proj_{L^\perp} \Omega) \supseteq \cl \cone(\Omega) \cap L^\perp$. Arbitrarily pick $x^* \in \cl \cone(\Omega) \cap L^\perp$. If $x^* \in \cone(\Omega)$, then $x^* = \proj_{L^\perp} x^* \in \proj_{L^\perp} \cone(\Omega) = \cone(\proj_{L^\perp} \Omega)$. If $x^i \rightarrow x^*$ for a sequence of $\{x^i\} \subseteq \cone(\Omega)$, then $\proj_{L^\perp} x^i \rightarrow x^*$ where $$\proj_{L^\perp} x^i \in \proj_{L^\perp} \cone(\Omega) = \cone(\proj_{L^\perp} \Omega).$$ Hence $x^* \in \cl \cone(\proj_{L^\perp} \Omega)$. This completes the proof.
\end{proof}

It is well-known that, a pointed, closed convex cone is a polyhedral cone, iff it has finitely many different extreme rays. See, e.g., Consequence 5 in \cite{MR1227751}.
Therefore, the main result of this section naturally follows.
\begin{reptheorem}{theo: equiv}
$\Omega$ is a set in $\R^{n+1}$ that contains $(0, \ldots, 0, 1)$, and
let $L = \lin(\cl \cone(\Omega))$. Then $\I(\Omega)$ is polyhedral iff $\cl \cone(\proj_{L^\perp} \Omega)$ has finitely many different extreme rays. 
\end{reptheorem}

\begin{proof}
By Corollary~\ref{coro: iff_polyhedral_cone} and Lemma~\ref{lem: easy_lemma}, we know that, $\I(\Omega)$ is polyhedral iff the pointed closed convex cone $\cl \cone(\proj_{L^\perp} \Omega)$ is polyhedral, and this is true iff $\cl \cone(\proj_{L^\perp} \Omega)$ has finitely many extreme rays. 
\end{proof}

This theorem shows that, in order to show the polyhedrality of closure $\I(\Omega)$, one only has to study the extreme ray of the pointed cone $\cl \cone(\proj_{L^\perp} \Omega)$. In the next section we are going to give a full characterization for those extreme rays, using points in $\Omega$.

\subsection{Characterization of Extreme Rays for the Closed Conical Hull}

First we define a new concept for the convergence of rays in a cone. 
\begin{definition}
Given a sequence $\{\alpha^i\}_{i \geq 1} \subseteq \R^n \setminus \{0\} $ and $\alpha^* \neq 0 \in \R^n$, if there exists $\{\lambda_i\}_{i \geq 1} > 0$ such that $\lim_{i \rightarrow \infty} \lambda_i  \alpha^i = \alpha^*$, then we say $\{\alpha^i\}$ \textbf{conically converges} to $\alpha^*$, or $\alpha^i \xrightarrow{c} \alpha^*$.
\end{definition}


Assume $\cl \cone(\Omega')$ is a pointed cone.
For any \textbf{extreme ray} $r \in \cl \cone(\Omega')$, there are two cases: $r \in \cone (\Omega')$, or $r \in \cl \cone (\Omega') \setminus \cone (\Omega')$. For the first case, it is not too hard to observe that $r \in (\Omega')_+$, since here $r$ is assumed to be an extreme ray. While for the second case, we know $r$ can be expressed as the limit of a convergent sequence in $\cone (\Omega')$, using the definition of \emph{conical convergence}, we know there exists $\{r^i\} \subseteq \conv (\Omega')$ such that $r^i \xrightarrow{c} r$. One of the main results in this section is that, this sequence $\{r^i\}$ can be further picked from $\Omega'$ instead of $\conv (\Omega')$. 
Note that $\cl(\cdot)$ and $\cone(\cdot)$ are not commutative, and this result will be trivial if we switch $\cl(\cdot)$ and $\cone(\cdot)$ with each other. 
We state those two cases formally as the following lemma:
\begin{lemma}
\label{lem: main}
$\Omega$ is a set in $\R^{n+1}$ and let $L = \lin(\cl \cone(\Omega))$. Then for any extreme ray $r \in \cl \cone(\proj_{L^\perp} \Omega)$, either $r \in (\proj_{L^\perp} \Omega)_+$, or there exist different $\{r^i\} \subseteq \Omega$ such that $\proj_{L^\perp} r^i \xrightarrow{c} r$. 
\end{lemma}

Henceforth, when we mention a ray $r$ of a cone, we will make no distinction between $r$ and its positive scalar multiplication. In other words, we say two rays $r^1$ and $r^2$ are different, iff there does not exist $\lambda > 0$ such that $r^1 = \lambda r^2$. 

Before proceeding to the proof of Lemma~\ref{lem: main}, we will require the following lemma.
\begin{lemma}[Supporting Hyperplane Theorem for pointed cone]
\label{lem: supporting_hyperplane_pointed}
Let $K \subseteq \R^n$ be a pointed closed convex cone. Then it is strictly supported at the origin: there is $h \in \R^n$ such that if $x \in K$ and $x \neq 0$, then $h^\transp x > 0$.
\end{lemma}

\begin{proof}
Since $K$ is pointed, we know the polar cone $K^\circ$ is full-dimensional. So we can find an interior point $x^* \in K^\circ$, which has: $ x^\transp x^*  < 0$ for all $x \neq 0 \in K$. Then simply picking $h = -x^*$ completes the proof.
\end{proof}

The next lemma is also well-known in literature.

\begin{lemma}[Lemma 2.4 in \cite{husseinov1999note}]
\label{lem: victor_klee_extreme}
Let $S$ be a non-empty closed set in $\R^n$. Then, every extreme point of $\cl \conv(S)$ belongs to $S$.
\end{lemma}

Now we are ready to prove Lemma~\ref{lem: main}. 
 \begin{proof}[Proof of Lemma~\ref{lem: main}]
 Let $\Omega_L: = \proj_{L^\perp} \Omega$. 
Lemma~\ref{lem: easy_lemma} says that, $\cl \cone(\Omega_L)$ is a pointed, closed convex cone, then from Lemma~\ref{lem: supporting_hyperplane_pointed}, we can find a supporting hyperplane $h^\transp x = 0$ such that $h^\transp \omega > 0$ for all $\omega \neq 0 \in \cl \cone(\Omega_L)$. 
Denote the normalized version of $\Omega_L$ as: $\Omega' := \{\frac{\omega}{h^\transp \omega} \mid \omega \in \Omega_L \setminus \{0\}\}$.
Here $\Omega'$ is well-defined since for all $\omega \neq 0 \in \Omega_L$ there is $h^\transp \omega > 0$. 
 \begin{claim}
 $\{x \in \R^n \mid h^\transp x  = 1\} \cap \cl \cone(\Omega_L) = \cl \conv (\Omega')$.
 \end{claim}
 \begin{cpf}
 First, we show  $\{x \in \R^n \mid h^\transp  x  = 1\} \cap \cl \cone(\Omega_L) \subseteq \cl \conv (\Omega')$. Arbitrarily pick $\alpha^*$ such that $h^\transp  \alpha^* = 1$, and there exist $\{\alpha^i\} \subseteq \cone(\Omega_L)$ such that $\alpha^i \rightarrow \alpha^*$. Denote $\beta^i: = \frac{\alpha^i}{h^\transp  \alpha^i}$. Since $\alpha^i \rightarrow \alpha^*, h^\transp  \alpha^* = 1$, we know $h^\transp  \alpha^i \rightarrow 1$. Hence we also have $\beta^i \rightarrow \alpha^*$, and here $\beta^i \in \{x \in \R^n \mid h^\transp x = 1\} \cap \cone(\Omega_L)$.  In the following, we show: $ \{x \in \R^n \mid h^\transp  x = 1\} \cap \cone (\Omega_L) \subseteq  \conv (\Omega')$, which will imply that $\alpha^* \in \cl \conv (\Omega')$ since $\beta^i \rightarrow \alpha^*$ and $\beta^i \in \{x \in \R^n \mid h^\transp x = 1\} \cap \cone(\Omega_L)$. According to the arbitrariness of $\alpha^* \in \{x \in \R^n \mid h^\transp x  = 1\} \cap \cl \cone(\Omega)$, this will complete the proof of $\{x \in \R^n \mid h ^\transp x  = 1\} \cap \cl \cone(\Omega_L) \subseteq \cl \conv (\Omega')$. 

Pick $\beta \in \{x \in \R^n \mid h^\transp  x = 1\} \cap \cone (\Omega_L),$ we can write it as: $\beta = \sum_{i=1}^k \lambda_i b^i$ for some $\lambda_i > 0, b^i \in \Omega_L, i \in [k], k \in \N$. Here because $\beta \in \{x \in \R^n \mid h^\transp  x = 1\} $, we know $\sum_{i=1}^k \lambda_i h^\transp b^i = 1$. Therefore, we can also write $\beta$ as:
$$
\beta = \sum_{i=1}^k ( \lambda_i h^\transp b^i) \cdot \frac{b^i}{h^\transp b^i}, \text{ here } \frac{b^i}{h^\transp b^i} \in \Omega',\ \sum_{i=1}^k \lambda_i h^\transp b^i = 1.
$$
We get $\beta \in \conv(\Omega')$, which concludes $ \{x \in \R^n \mid h^\transp  x = 1\} \cap \cone (\Omega_L) \subseteq  \conv(\Omega')$.

Lastly, we show the other direction $\{x \in \R^n \mid h^\transp  x = 1\} \cap \cl \cone(\Omega_L) \supseteq \cl \conv (\Omega')$. By definition, $\Omega' \subseteq \{x \in \R^n \mid h^\transp  x = 1\}$, which implies $\cl \conv (\Omega') \subseteq \{x \in \R^n \mid h^\transp  x = 1\}$. On the other hand, clearly $\Omega' \subseteq \cone(\Omega_L)$, so $\cl \conv (\Omega')\subseteq \cl \cone (\Omega_L)$, and we complete the proof for this claim. 
\end{cpf}
Given an extreme ray $r$ of $\cl \cone(\Omega_L)$, w.l.o.g. we assume that $h^\transp  r = 1$.
Then $r \in \Omega_L$ iff $r \in \Omega'$.
From the above claim, we also know $r \in \cl \conv (\Omega')$.
\begin{claim}
$r$ is an extreme point of $\cl \conv (\Omega')$.
\end{claim}
\begin{cpf}
Assuming $r = \sum_{i=1}^k \lambda_i a^i$ for $\lambda_i > 0, \sum_{i=1}^k \lambda_i = 1$ and $r \neq a^i \in \cl \conv (\Omega')$. From the definition of $\Omega'$, we also have $h^\transp a^i = 1, a^i \in \cl \cone(\Omega_L)$. According to the extreme ray assumption of $r$, while it can be written as the conical combination (convex combination is also conical combination) of other points in $\cl \cone(\Omega_L)$, we know there exists $\gamma_i > 0$ such that $a^i = \gamma_i r$. Since $h^\transp  r = h^\transp  a^i = 1$, we have $\gamma_i = 1$, meaning $a^i = r$, which contradicts the assumption that $r \neq a^i$. 
\end{cpf}
So for any extreme ray $r \in \cl \cone(\Omega_L)$ with $h^\transp r = 1$, $r$ is an extreme point of $\cl \conv (\Omega')$. It is obvious that $\cl \conv (\Omega') = \cl \conv( \cl ({\Omega'}))$, 
so $r$ is an extreme point of $\cl \conv(\cl ({\Omega'}))$. By Lemma~\ref{lem: victor_klee_extreme}, we obtain that $r \in \cl ({\Omega'})$. By definition of $\Omega'$, it implies that either $r  \in (\Omega_L)_+$, or there exists different $\{r^i\} \subseteq \Omega_L$ such that $r^i \xrightarrow{c} r$. This completes the proof.
 \end{proof}

\section{$k$-aggregation Closure of Covering Sets}
\label{sec: 3}

In the last section, we have shown that, closure $\I(\Omega)$ is a polyhedron iff $\cl \cone(\proj_{L^\perp} \Omega)$ has finitely many extreme rays (Theorem~\ref{theo: equiv}), where each extreme ray can be characterized by points in $\proj_{L^\perp} \Omega$ (Lemma~\ref{lem: main}). 
As an application to these results, in this section, we are going to establish the polyhedrailty of the $k$-aggregation closure for covering sets. For detailed definition of $k$-aggregation closure $\mathcal{A}_k(Q)$ for a covering set $Q$, see \eqref{eq: defn_generalized_k_agg} in the Introduction Section.

The well-known Dickson's Lemma will be used in our later proof, and it also played an important role in some other relevant papers, see, e.g., \cite{MR2969261,MR4207341,pashkovich2021aggregation}. Here we mention two equivalent statements for this lemma. 
\begin{lemma}[Dickson's Lemma \cite{dickson1913finiteness}]
\label{lem: Dickson_lemma}
For any $X \subseteq \N^n$:
\begin{enumerate}
\item  there exists $X' \subseteq X$ with $|X'| < \infty$, such that every $x \in X$ satisfies $x' \leq x$ for some $x' \in X'$.
\item The partially-ordered set (poset) $(X, \leq)$ has no infinite antichain.
\end{enumerate}
\end{lemma}

In order theory, an \emph{antichain} is a subset of a poset such that any two distinct elements in the subset are incomparable.
From Dickson's Lemma we have the following proposition.

\begin{proposition}
\label{prop: integerl_hull_covering}
For any covering polyhedron $C$, its integer hull $\conv(C \cap \Z^n)$ is also a covering polyhedron.
\end{proposition}

\begin{proof}
Denote $C_I: = C \cap \Z^n$. Since $C_I \subseteq \N^n$, by Dickson's Lemma~\ref{lem: Dickson_lemma}, we know there exists a finite subset $C_I' \subseteq C_I$, such that every $x \in C_I$ satisfies $x' \leq x$ for some $x' \in C_I'$. Next we show: $\conv(C_I') + \R^n_+ = \conv(C_I)$.

Since $C'_I \subseteq C_I$, we know $\conv(C_I') + \R^n_+ \subseteq \conv(C_I) + \R^n_+ = \conv(C_I)$, here the last equality is because that $\R^n_+$ is the recession cone of $C$ and the integer hull $\conv(C_I)$. On the other hand, for any $x \in \conv(C_I)$, it can be written as the convex combination of points in $C_I$, meaning that there exists $t_j \geq 0, y_j \in C_I, j \in J, \sum_{j \in J} t_j = 1$, such that $x = \sum_{j \in J} t_j y_j$. From the construction of $C_I'$, we know for each $y_j \in C_I,$ there is $y_j' \leq y_j$ for some $y_j' \in C_I'$. So: $\sum_{j \in J} t_j y_j' \leq \sum_{j \in J} t_j y_j = x$, and here $\sum_{j \in J} t_j y_j' \in \conv(C_I')$. Hence $x \in \conv(C_I') + \R^n_+$, which concludes the proof of $\conv(C_I') + \R^n_+ = \conv(C_I)$.

So far we have shown that the integer hull $\conv(C_I)$ is a polyhedron, due to the finiteness of $C_I'$. Denote $\conv(C_I): = \{x \in \R^n_+ \mid Ax \leq b\}$. Lastly, we want to show that $(A,b)$ can be assumed to have non-positive entries. This is obvious since if some entry $A_{ij} > 0$, then $e^j$ will not be an extreme ray of the recession cone of $\conv(C_I)$. For the right-hand-side vector $b$, if some entry is positive, then its corresponding constraint will become redundant since $\conv(C_I) \subseteq \R^n_+$. Therefore, we complete the proof.
\end{proof}

Given a covering set $Q$, for the $\Lambda(Q)$ as defined in \eqref{eq: Lambda_S}, we have the following easy observation:
\begin{observation}
$\Lambda(Q) \subseteq \R^n_-$.
\end{observation}

Now we denote $\mathcal{C}^k(Q)$ to be the set of all covering polyhedron with $k$ valid inequalities of $Q$:
$$
\mathcal{C}^k(Q) = \Big\{\{x \in \R^n_+ \mid (f^j)^\transp x \leq \sup\{(f^j)^\transp y \mid y \in Q\}, \ \forall j \in [k] \} \mid f^1, \ldots, f^k \in \Lambda(Q)  \Big\}.
$$

For any $C \in \mathcal{C}^k(Q)$, let
$$
\mathscr{F}_C: = \{(\alpha, \beta)  \mid \alpha^\transp x \leq \beta \text{ defines a facet of }\conv(C \cap \Z^n), \|(\alpha, \beta)\| = 1\}.
$$
Here the additional constraint of $\|(\alpha, \beta)\| = 1$ is to ensure that $\mathscr{F}_C$ is a finite set instead of a cone, and the norm $\|\cdot\|$ here refers to the Euclidean norm.
By Proposition~\ref{prop: integerl_hull_covering}, we know $\mathscr{F}_C \subseteq \R^{n+1}_-$.
Lastly, denote the family of $k$-aggregation cuts as:
\begin{equation}
\label{eq: Omega_k_aggregation_covering}
\Omega_Q: = \{(-e^j, 0)\}_{j \in [n]} \bigcup \{(0, \ldots, 0, 1)\}  \bigcup_{C \in \mathcal{C}^k(Q)} \F_C.
\end{equation}
Here $(-e^j, 0)$ corresponds to the non-negative constraint $x_j \geq 0$. 
According to our definition in \eqref{eq: defn_generalized_k_agg}, we have $\I(\Omega_Q) = \mathcal{A}_k (Q)$.

The next two propositions will be the corner stones to prove the main Theorem~\ref{theo: aggregation_sign_pattern}.

\begin{proposition}
\label{prop: first_step}
$\cl \cone(\Omega_Q) = \cone(\Omega_Q)$. 
\end{proposition}

\begin{proposition}
\label{prop: second_step}
$\cone(\Omega_Q)$ has only finitely many different extreme rays. 
\end{proposition}

Given these two results and the fact that covering set $Q$ always has full-dimensional integer hull ($\R^n_+$ is the recession cone), then the main Theorem~\ref{theo: aggregation_sign_pattern} follows directly from Theorem~\ref{theo: equiv}.

\begin{proof}[Proof of Theorem~\ref{theo: aggregation_sign_pattern}]
Let $L = \lin(\cl \cone(\Omega_Q)) \subseteq \R^{n+1}$. Arbitrarily pick a non-zero $(\alpha, a_0) \in L$, from definition of the linearity space, we have $(\alpha, a_0) \in \cl \cone(\Omega_Q)$ and $-(\alpha, a_0) \in \cl \cone(\Omega_Q)$. 
By Proposition~\ref{prop: valid_ineq_for_closure}, we know that $\alpha^\transp x \leq a_0$ and $-\alpha^\transp x \leq -a_0$ are both valid inequalities to $\I(\Omega_Q)$, which is $\mathcal{A}_k(Q)$. This implies that $\mathcal{A}_k(Q)$ is contained in the hyperplane given by $a^\transp x = a_0$. However, a covering set $Q$ obviously has full-dimensional integer hull, this gives the contradiction. Therefore, $L = \{0\}$, and Theorem~\ref{theo: equiv} tells: $\mathcal{A}_k(Q)$ is polyhedral iff $\cl \cone(\Omega)$ has finitely many different extreme rays. Then Proposition~\ref{prop: first_step} together with Proposition~\ref{prop: second_step} complete the proof.
\end{proof}

Next, we are devoting the following two subsections to the proofs of these two propositions individually.

\subsection{Proof of Proposition~\ref{prop: first_step}}

In order to prove Proposition~\ref{prop: first_step}, first we will need the following lemma. For a finite set $X \subseteq \N^n$, we let $\max(X)$ denote the vector $(\max\{x_1 \mid x \in X\}, \ldots, \max\{x_n \mid x \in X\})$. 
\begin{lemma}
    \label{lem: projection}
Given a covering set $Q \subseteq \R^n_+$ and $k \in \N$, for any $t < n, t \in \N$, there exists another covering set $Q' \subseteq \R^t_+$ such that $\proj_{[t]} \mathcal{A}_k(Q) = \mathcal{A}_k(Q')$. Moreover, for any $C' \in \mathcal{C}(Q')$, there exists $C \in \mathcal{C}^k(Q)$ such that $\proj_{[t]} \conv(C \cap \Z^n) = \conv(C' \cap \Z^t)$. 
    \end{lemma}
    
    \begin{proof}
    For a covering set $Q$, by the definition of $k$-aggregation closure in \eqref{eq: defn_generalized_k_agg}, we know $\mathcal{A}_k(Q) = \bigcap_{C \in \mathcal{C}^k(Q)} \conv(C \cap \Z^n)$. First, we show that $\proj_{[t]} \mathcal{A}_k(Q) = \bigcap_{C \in \mathcal{C}^k(Q)} \proj_{[t]} \conv(C \cap \Z^n)$.
    
    The $\subseteq$ relation is obvious, so we only have to show that, for any $x' \in \bigcap_{C \in \mathcal{C}^k(Q)} \proj_{[t]} \conv(C \cap \Z^n),$ there is also $x' \in \proj_{[t]} \mathcal{A}_k(Q)$. By assumption, we know for any $C \in \mathcal{C}^k(Q)$, there exists $y^C \in \R^{n-t}_+$ such that $(x', y^C) \in \conv(C \cap \Z^n)$. In other words, for any $C \in \mathcal{C}^k(Q)$, there exists $\lambda_{C, i} \geq 0$ for $i \in [n+1]$, $(x^{C,i}, y^{C,i}) \in C \cap \Z^n, \sum_{i=1}^{n+1} \lambda_{C,i} = 1$, such that $\sum_{i=1}^{n+1} \lambda_{C,i} (x^{C,i}, y^{C,i}) = (x', y^C)$. Now take $y^*: = \max(\{y^{C,i} \mid C \in \mathcal{C}^k(Q), i \in [n+1]\})$ (see our denotation ahead of this lemma). 
Since each $C \in \mathcal{C}^k(Q)$ is a covering polyhedron and $(x^{C,i}, y^{C,i}) \in C \cap \Z^n,$ so there is also $(x^{C,i}, y^*) \in C \cap \Z^n$. Hence $(x', y^*) = \sum_{i=1}^{n+1} \lambda_{C,i} (x^{C,i}, y^*) \in \conv(C \cap \Z^n)$. Therefore, $(x', y^*) \in \bigcap_{C \in \mathcal{C}^k(Q)} \conv(C \cap \Z^n)$, which implies that $x' \in \proj_{[t]} \mathcal{A}_k(Q)$. 
    
For any $C \in \mathcal{C}^k(Q),$ there is also $\proj_{[t]} \conv(C \cap \Z^n) =\conv \big( \proj_{[t]} (C \cap \Z^n) \big)$. Next we show that $\proj_{[t]} (C \cap \Z^n) = \proj_{[t]} C  \cap \Z^t$. The inclusion $\subseteq$ is obvious. For any $x' \in \proj_{[t]} C  \cap \Z^t$, say $(x', y') \in C$, then there is $(x', \lceil y' \rceil) \in C \cap \Z^n$, which implies that $x' \in \proj_{[t]} (C \cap \Z^n)$. Hence $\proj_{[t]} (C \cap \Z^n) = \proj_{[t]} C  \cap \Z^t$. 

So far we have obtained that $\proj_{[t]} \mathcal{A}_k(Q) = \bigcap_{C \in \mathcal{C}^k(Q)} \conv(\proj_{[t]} C \cap \Z^t)$, and for any $C \in \mathcal{C}^k(Q)$, $\proj_{[t]} (C \cap \Z^n) = \proj_{[t]} C  \cap \Z^t$. 
Let $Q': = \proj_{[t]} Q$, which is a covering set in $\R^t_+$. Lastly it suffices for us to show $\mathcal{C}(Q') = \{\proj_{[t]} C \mid C \in \mathcal{C}^k(Q)\}$. Since if this is true, then $\proj_{[t]} \mathcal{A}_k(Q) = \bigcap_{C \in \mathcal{C}^k(Q)} \conv(\proj_{[t]} C \cap \Z^t) = \bigcap_{C' \in \mathcal{C}(Q')} \conv(C' \cap \Z^t) = \mathcal{A}_k(Q')$, and 
for any $C' \in \mathcal{C}(Q'), \conv(C' \cap \Z^t) = \conv(\proj_{[t]} C  \cap \Z^t) = \conv \big(\proj_{[t]} (C \cap \Z^n) \big) = \proj_{[t]} \conv(C \cap \Z^n) $ for some $C \in \mathcal{C}^k(Q)$. These complete the proof of this lemma.

Arbitrarily pick $C \in \mathcal{C}^k(Q)$, it can be written as $C = \{x \in \R^n_+ \mid (f^j)^\transp x \leq g_j, \ \forall j \in [k]\}$, here each $(f^j)^\transp x \leq g_j$ is valid to $Q,$ and $(f^j, g_j) \in (\R^n_-, \R_-)$. 
Define $J_0: = \{j \in [k] \mid f^j_i = 0, \ \forall i \in [n] \setminus [t]\}$, and $J_1 : = [k] \setminus J_0$. 
Now for this particular $C$, define $C': = \{x \in \R^t_+ \mid (f^j_{[t]})^\transp x \leq g_j \ \forall j \in J_0, \textbf{0}^\transp x \leq 0 \ \forall j \in J_1\}$. Since $(f^j)^\transp x \leq g_j$ is valid to $Q$, we know that $(f^j_{[t]})^\transp x \leq g_j$ is valid to $\proj_{[t]}Q$, which is just $Q'$. Hence $C' \in \mathcal{C}(Q')$. 
Moreover, we have the following claim:
\begin{claim}
$\proj_{[t]} C = C'$.
\end{claim}
\begin{cpf}
$\proj_{[t]} C \subseteq C'$ is obvious. Now arbitrarily pick $x' \in C'$, which satisfies $ (f^j_{[t]})^\transp x' \leq g_j \ \forall j \in J_0$, we want to show that $x' \in \proj_{[t]} C$. Consider point $x^* : = (x', N, \ldots, N)$ for some large enough number $N$. Then by definition of $J_0, (f^j)^\transp x^* = (f^j_{[t]})^\transp x' \leq g_j$ for any $j \in J_0$, and $(f^j)^\transp x^* < g_j$ for any $j \in J_1$ since $N$ is picked as a large number. Hence $x^* \in C$, which implies that $x' \in \proj_{[t]} C$.
\end{cpf}
From this above claim, we have shown that, for any $C \in \mathcal{C}^k(Q)$, there exists $C' \in \mathcal{C}(Q')$, such that $\proj_{[t]} C = C'$. Hence $\mathcal{C}(Q') \supseteq \{\proj_{[t]} C \mid C \in \mathcal{C}^k(Q)\}$. Now for any $C' \in \mathcal{C}(Q')$, we can write it as $\{x \in \R^t_+ \mid (f'^j)^\transp x \leq g'_j, \ \forall j \in [k]\}$, and we want to show that $C' = \proj_{[t]} C$ for some $C \in \mathcal{C}^k(Q)$. Let $C: = \{x \in \R^n_+ \mid \left((f'^j)^\transp, 0, \ldots, 0 \right) x \leq g'_j, \ \forall j \in [k]\}$, it is easy to see that each $\left((f'^j)^\transp, 0, \ldots, 0 \right) x \leq g'_j$ is valid to $Q$, hence $C \in \mathcal{C}^k(Q)$. Notice that $\proj_{[t]} C = C'$, therefore, we have also shown that $\mathcal{C}(Q') \subseteq \{\proj_{[t]} C \mid C \in \mathcal{C}^k(Q)\}$, completing the proof.
\end{proof}

Now we are ready to prove Proposition~\ref{prop: first_step}.

\begin{proof}[Proof of Proposition~\ref{prop: first_step}]
Notice that $\Omega_Q \subseteq \{x \in \R^{n+1} \mid \|x\| = 1\}$, 
it is therefore equivalent of showing that: for any extreme ray $r$ of $\cl \cone(\Omega_Q)$ with $\|r\| = 1$, there is $r \in \Omega_Q$. 
We prove by induction on dimension. When $n=1$ the result can be trivially verified. Next we assume that the statement of this proposition holds when $n \leq N$, and we consider the case of $n = N+1$. 

Assume for contradiction that $(\alpha^*, \beta^*) \in \cl \cone(\Omega_Q)$ is an extreme ray with $\|(\alpha^*, \beta^*)\| = 1$ and $(\alpha^*, \beta^*) \notin \Omega_Q$. Then by Lemma~\ref{lem: main}, we know there exist different $(\alpha^i, \beta_i) \in \F_{C_i}$ for $C_i \in \mathcal{C}^k(Q), \gamma_i > 0, i \in \N$, such that $\gamma_i (\alpha^i, \beta_i) \rightarrow (\alpha^*, \beta^*)$. 
By Proposition~\ref{prop: integerl_hull_covering}, here each $(\alpha^i, \beta_i) \in \R^{n+1}_-$, so $(\alpha^*, \beta^*) \in \R^{n+1}_-.$
Here w.l.o.g. we further assume that $\beta_i < 0$ for each $i \in \N$, since $(\alpha^i, \beta_i)$ with $\beta_i = 0$ reduces to the trivial non-negative constraints. 
Next we argue by two cases, depending on whether $\supp(\alpha^*) = [n]$ or not. In either case we want to establish the contradiction.

\smallskip \noindent
\textbf{Case 1}: $\supp(\alpha^*) = [n]$.
In this case, 
for each $j \in [n]$ there is $\alpha^*_j < 0$, and we have 
$\{x \in \R^n_+ \mid (\alpha^*)^\transp x \geq \beta^* - 1\}$ is a bounded set, so $\{x \in \N^n \mid (\alpha^*)^\transp  x \geq \beta^* - 1\}$ is a finite set. 
Since $\gamma_i (\alpha^i, \beta_i) \rightarrow (\alpha^*, \beta^*)$, we know there exists $n_0$ such that when $i \geq n_0, \{x \in \N^n \mid (\gamma_i \alpha^i)^\transp x = \gamma_i \beta_i\} \subseteq \{x \in \N^n \mid (\alpha^*)^\transp x \geq \beta^* - 1\}$. By the pigeonhole principle, we know there must exist $(\alpha^a, \beta_a)$ and $(\alpha^b, \beta_b)$ with $\{x \in \N^n \mid  (\alpha^a)^\transp x =  \beta_a\} = \{x \in \N^n \mid  (\alpha^b)^\transp x =  \beta_b\}$. Since $(\alpha^a)^\transp x \leq \beta_a$ is facet-defining inequality of $\conv(C_a \cap \Z^n)$ whose extreme points are all integral, and hyperplane $(\alpha^a)^\transp x = \beta_a$ does not pass through the origin ($\beta_a < 0$), therefore, we can find $n$ linearly independent integer points in $\{x \in \N^n \mid  (\alpha^a)^\transp x =  \beta_a\}$, say $v^1, \ldots, v^n$. consider the linear system
\begin{equation}
\label{eq: linear_system}
\left(\begin{array}{cc}(v^1)^\transp, & -1 \\ \vdots & \vdots \\(v^n)^\transp, & -1\end{array}\right) \cdot x = \left(\begin{array}{c}0 \\ \vdots \\0\end{array}\right).
\end{equation}
Since $v^1, \ldots, v^n$ are linearly independent, we know that the solution set of \eqref{eq: linear_system} forms a 1-dimensional linear space. However, both $(\alpha^a, \beta_a)^\transp$ and $(\alpha^b, \beta_b)^\transp$ are solutions of \eqref{eq: linear_system}, together with the condition that $\|(\alpha^a, \beta_a)\| = \|(\alpha^b, \beta_b)\| = 1$ and $(\alpha^a, \beta_a), (\alpha^b, \beta_b) \in \R^{n+1}_-$, we obtain that $(\alpha^a, \beta_a) = (\alpha^b, \beta_b)$, which contradicts to the initial assumption that all $(\alpha^i, \beta_i)$ are different.

\smallskip \noindent
\textbf{Case 2}: $\supp(\alpha^*) \subset [n]$.
W.l.o.g. we assume $\supp(\alpha^*) = [t]$ for $t < n$.
    By Lemma~\ref{lem: projection}, there exists $Q' \subseteq \R^t_+$ such that $\proj_{[t]} \mathcal{A}_k(Q) = \mathcal{A}_k(Q')$. 
    Since $\alpha^*_{[t]} x \leq \beta^*$ is valid to $\proj_{[t]} \mathcal{A}_k(Q)$, it will also be valid to $\mathcal{A}_k(Q')$. 
    For $Q' \subseteq \R^t_+$, we denote its corresponding family of $k$-aggregation cuts as in \eqref{eq: Omega_k_aggregation_covering} to be $\Omega_{Q'}$. Then $\mathcal{A}_k(Q') = \I(\Omega_{Q'})$, and by Proposition~\ref{prop: valid_ineq_for_closure}, we know $(\alpha^*_{[t]}, \beta^*) \in \cl \cone(\Omega_{Q'})$. Now we argue that $(\alpha^*_{[t]}, \beta^*)$ is actually an extreme ray of $\cl \cone(\Omega_{Q'})$. Assume that $(\alpha^*_{[t]}, \beta^*)$ can be written as the conical combination of some different rays $(\alpha'^j, \beta'_j) \in \cl \cone(\Omega_{Q'})$ for $j \in J$, then each $(\alpha'^j)^\transp x \leq \beta'_j$ is valid to $\mathcal{A}_k(Q')$, which yields that each $(\alpha'^j, 0, \ldots, 0) x \leq \beta'_j$ is valid to $\mathcal{A}_k(Q)$ since $\proj_{[t]} \mathcal{A}_k(Q) = \mathcal{A}_k(Q')$. Hence by Proposition~\ref{prop: valid_ineq_for_closure}, $(\alpha'^j, 0, \ldots, 0, \beta'_j) \in \cl \cone(\Omega_Q)$, for all $j \in J.$ Therefore, the extreme ray $(\alpha^*, \beta^*) = (\alpha^*_{[t]}, 0, \ldots, 0, \beta^*)$ can be written as the conical combination of $(\alpha'^j, 0, \ldots, 0, \beta'_j) \in \cl \cone(\Omega_Q), j \in J$, which gives the contradiction. 
So $(\alpha^*_{[t]}, \beta^*)$ is an extreme ray of $\cl \cone(\Omega_{Q'})$.
By inductive hypothesis and $Q'$ is in dimension $\R^t$, we know $(\alpha^*_{[t]}, \beta^*) \in \Omega_{Q'}$. Say $(\alpha^*_{[t]})^\transp x \leq \beta^*$ is a facet-defining inequality of some $\conv(C' \cap \Z^t)$ for $C' \in \mathcal{C}(Q')$. By Lemma~\ref{lem: projection}, we know there exists $C \in \mathcal{C}^k(Q)$, such that $\proj_{[t]} \conv(C \cap \Z^n) = \conv(C' \cap \Z^t)$. Hence $(\alpha^*)^\transp x = (\alpha^*_{[t]}, 0,\ldots, 0) x \leq \beta^*$ is valid to $\conv(C \cap \Z^n)$. Since $(\alpha^*, \beta^*)$ is assumed to be an extreme ray of $\cl \cone(\Omega_Q)$, by Proposition~\ref{prop: valid_ineq_for_closure}, $(\alpha^*)^\transp x \leq \beta^*$ has to be a facet-defining inequality of $\conv(C \cap \Z^n)$, that is $(\alpha^*, \beta^*) \in \F_{C} \subseteq \Omega_Q$, contradicting the foremost assumption of $(\alpha^*, \beta^*) \notin \Omega_Q$. 

Therefore, when the dimension is $N+1$, the statement of this proposition also holds. By induction we conclude the proof.
\end{proof}

\subsection{Proof of Proposition~\ref{prop: second_step}}

To prove Proposition~\ref{prop: second_step}, we will need the following lemma which can be seen as a obvious consequence from Theorem~1 in \cite{MR4207341}, and as a generalization of Dickson's Lemma~\ref{lem: Dickson_lemma}. For the completeness of this paper, we also include its proof here.
\begin{lemma}
\label{lem: wqo_easy}
Given an infinite set $\mathcal{S}$ whose elements are subsets in $\N^n$. Then there must exist $S_1, S_2 \in \mathcal S,$ such that 
$$\{x \in \N^n \mid \exists s_1 \in S_1 \ \st \ x \leq s_1\} \subseteq \{x \in \N^n \mid \exists s_2 \in S_2 \ \st \ x \leq s_2\}.$$
\end{lemma}

\begin{proof}
Theorem~1 in \cite{MR4207341} states that: The set of integer packing sets in $\R^n$ is well-quasi-ordered by the relation $\subseteq$. Here an integer packing set in $\R^n$ is a subset $X \in \N^n$ with the property that: if $x \in X, y \in \N^n$ and $y \leq x$, then $y \in X$. Then $\{\{x \in \N^n \mid \exists s \in S \ \st \ x \leq s\} \ \forall S \in \mathcal S\}$ is clearly an infinite set of integer packing sets, which is well-quasi-ordered by the relation $\subseteq$. In order theory, one necessary condition for an infinite poset $(X, \prec)$ to be a well-quasi-ordered set is that, $\exists x_1, x_2 \in X, \st \ x_1 \prec x_2$. Hence we finish the proof.
\end{proof}

Now we present the main proof of this subsection.

\begin{proof}[Proof of Proposition~\ref{prop: second_step}]
Assuming for contradiction that $\{(\alpha^i, \beta_i)\}_{i \geq 1}$ is a sequence of infinitely many different extreme rays of $\cone(\Omega_Q)$.
Here w.l.o.g. we assume that $\|(\alpha^i, \beta_i)\| = 1$. Since $(\alpha^i, \beta_i)$ is extreme ray of $\cone(\Omega_Q)$ for each $i \geq 1$, and $\Omega_Q \subseteq \{x \in \R^{n+1} \mid \|x\| = 1\}$, we further know that $(\alpha^i, \beta_i) \in \Omega_Q$. 
Assuming each $(\alpha^i)^\transp x \leq \beta_i$ is facet-defining inequality of $\conv(C_i \cap \Z^n)$ for some $C_i \in \mathcal{C}^k(Q)$, where $\beta_i < 0, \alpha^i \in \R^n_-$. 
Within this infinite sequence $\{(\alpha^i, \beta_i)\}_{i \geq 1}$, by the pigeonhole principle, there exist infinitely many $\alpha^i$  with the same support $\supp(\alpha^i)$. W.l.o.g. we assume for all $i \in \N, \supp(\alpha^i) = [t]$ for some $t \leq n$. Next, we argue that there must exist $i^*, j^* \in \N$, such that $(\alpha^{i^*})^\transp x \leq \beta_{i^*}$ is not facet-defining to $\conv(C_{i^*} \cap \Z^n) \cap \conv(C_{j^*} \cap \Z^n)$.

For each $i \in \N$, denote by $E_i$ the set of extreme points of the facet $F_i: = \conv(C_i \cap \Z^n) \cap \{x \in \R^n \mid (\alpha^i)^\transp x = \beta_i\}$. Obviously the extreme rays each $F_i$ are $e^{t+1}, \ldots, e^n$, so $F_i = \conv(E_i) + \{(0,\ldots, 0)\} \times \R^{n-t}_+$. Let $\{E_i \mid  i \in \N\}$ be the infinite set $\mathcal S$ in Lemma~\ref{lem: wqo_easy}, and from such lemma, we know there exists $i^*$ and $j^*$, such that $\{x \in \N^n \mid \exists y \in E_{i^*} \ \st \ x \leq y\} \subseteq \{x \in \N^n \mid \exists y \in E_{j^*} \ \st \ x \leq y\}$. 
To help visualize, see Figure~\ref{fig: last_fig}.
Specifically, for any $x \in E_{i^*}$, there exists $y \in E_{j^*}$ such that $x \leq y$. Note that:
\begin{align}
\label{eq: prop3_proof}
\begin{split}
& \{x \in \R^n \mid (\alpha^{i^*})^\transp x = \beta_{i^*}\} \bigcap \big(\conv(C_{i^*} \cap \Z^n) \cap \conv(C_{j^*} \cap \Z^n)\big) \\
= & \big(\{x \in \R^n \mid (\alpha^{i^*})^\transp x = \beta_{i^*}\} \cap \conv(C_{i^*} \cap \Z^n)\big) \bigcap \conv(C_{j^*} \cap \Z^n)\\
= & F_{i^*} \cap \conv(C_{j^*} \cap \Z^n) \\
\subseteq & F_{i^*} \cap \{x \in \R^n \mid (\alpha^{j^*})^\transp x \leq \beta_{j^*}\}.
\end{split}
\end{align}

Here $F_{i^*} = \conv(E_{i^*}) + \{(0,\ldots, 0)\} \times \R_+^{n-t}$. For any $x^{i^*} \in E_{i^*}$, since there exists $x^{j^*} \in E_{j^*}$ such that $x^{i^*} \leq x^{j^*}$, and here $(\alpha^{j^*})^\transp x^{j^*} = \beta_{j^*}, \alpha^{j^*} \in \R^n_-$, we know that: $(\alpha^{j^*})^\transp x^{i^*} \geq (\alpha^{j^*})^\transp x^{j^*} = \beta_{j^*}$. Moreover, since $\supp(\alpha^{j^*}) = [t]$, so for any $x \in \{(0,\ldots, 0)\} \times \R_+^{n-t}, (\alpha^{j^*})^\transp x = 0$. Therefore:
\begin{align*}
\begin{split}
F_{i^*} \cap \{x \in \R^n \mid (\alpha^{j^*})^\transp x \leq \beta_{j^*}\} & = F_{i^*} \cap \{x \in \R^n \mid (\alpha^{j^*})^\transp x = \beta_{j^*}\} \\ & \subseteq \{x \in \R^n \mid (\alpha^{j^*})^\transp x = \beta_{j^*}, (\alpha^{i^*})^\transp x = \beta_{i^*}\}.
\end{split}
\end{align*}
Combined with \eqref{eq: prop3_proof}, we know $\{x \in \R^n \mid (\alpha^{i^*})^\transp x = \beta_{i^*}\} \bigcap \big(\conv(C_{i^*} \cap \Z^n) \cap \conv(C_{j^*} \cap \Z^n)\big)$ has dimension at most $n-2$, which means that $(\alpha^{i^*})^\transp x \leq \beta_{i^*}$ cannot be a facet-defining inequality for $\conv(C_{i^*} \cap \Z^n) \cap \conv(C_{j^*} \cap \Z^n)$.

Since $(\alpha^{i^*})^\transp x \leq \beta_{i^*}$ is valid for $\conv(C_{i^*} \cap \Z^n) \cap \conv(C_{j^*} \cap \Z^n)$, we know there must exist different valid inequalities $(u^i)^\transp x \leq v_i, i \in [\ell], \ell \in \N$, for $\conv(C_{i^*} \cap \Z^n) \cap \conv(C_{j^*} \cap \Z^n)$, such that $(\alpha^{i^*}, \beta_{i^*})$ can be written as the conical combination of $(u^1, v_1), \ldots, (u^\ell, v_{\ell})$. Note that for each $i \in [\ell], (u^i)^\transp x \leq v_i$ is also valid for $\I(\Omega_Q)$, by Proposition~\ref{prop: valid_ineq_for_closure}, $(u^i, v_i) \in \cl \cone(\Omega_Q) \ \forall i \in [\ell]$, which is the same set as $\cone(\Omega_Q)$ due to Proposition~\ref{prop: first_step}. Therefore, $(\alpha^{i^*}, \beta_{i^*})$ can be written as the conical combination of finitely many different rays in $\cone(\Omega_Q)$, which contradicts the original assumption that $(\alpha^{i^*}, \beta_{i^*})$ is an extreme ray of $\cone(\Omega_Q)$. 
\end{proof}

\begin{figure}
    \centering
\begin{tikzpicture}[scale=1.0]
\draw [<->,thick] (0,5) node (yaxis) [above] {$x_2$}
        |- (5,0) node (xaxis) [right] {$x_1$};
    \draw[thick, blue] (0.3,1.5) -- (1.5,0.2);
    \draw[thick, blue] (0.6,2.0) -- (2.0,0.6);
    \draw[thick, red] (0.3, 1.5) -- (-0.9, 2.8);
    \draw[thick, red] (1.5,0.2) -- (2.7,-1.1) node [right] {$\alpha^{i^*} x = \beta_{i^*}$};
    \draw[thick, red] (0.6,2.0) -- (-0.5, 3.1);
    \draw[thick, red] (2.0,0.6) -- (3.1, -0.5) node [right] {$\alpha^{j^*} x = \beta_{j^*}$};
    \draw (0.3, 1.5) -- (0, 2.5);
    \draw (1.5,0.2) -- (2.4, 0);
    \draw (0.6, 2.0) -- (0, 3.0);
        \draw (2.0, 0.6) -- (3.5,0);
     \fill[red] (0.3,1.5) circle (1pt);
      \fill[red] (1.5,0.2) circle (1pt);
      \fill[red] (0.6,2.0) circle (1pt);
      \fill[red] (2.0,0.6) circle (1pt);
      \fill[yellow, opacity = 0.2] (0,5) -- (0, 3.0) -- (0.6, 2.0) -- (2.0, 0.6) -- (3.5,0) -- (5, 0) -- (5,5);
      \fill[orange, opacity = 0.2] (0,5) -- (0, 2.5) -- (0.3, 1.5) -- (1.5,0.2) -- (2.4, 0) -- (5, 0) -- (5,5);
      \draw[blue] (1.3,1.3) node [above] {$F_{j^*}$};
      \draw[blue] (0.9, 0.9) node [above] {$F_{i^*}$};
            \draw[red] (2.5,2.1) node [above] {$\conv(C_{j^*} \cap \Z^n)$};
\end{tikzpicture}
    \caption{Here $\{x \in \N^n \mid \exists y \in E_{i^*} \ \st \ x \leq y\} \subseteq \{x \in \N^n \mid \exists y \in E_{j^*} \ \st \ x \leq y\}$,
    where $E_{i^*}$ (and $E_{j^*})$ is the set of two extreme points of facet $F_{i^*}$ (and $F_{j^*})$.}
    \label{fig: last_fig}
\end{figure}

\section{Conclusion}

In this paper, we propose a novel way of showing the polyhedrality of general cutting-plane closure, by characterizing the extreme rays of one specific pointed closed convex cone. Using this proof technique we close the remaining open problem left in Bodur et~al. \cite{bodur2017aggregation}, showing that the $k$-aggregation closure of a non-polyhedral covering set is still a polyhedron.
Due to the lack of specificity for our proof technique, this is not a one-size-fits-all approach, and in most cases further argument is inevitable. 
That being said, we do believe that such approach can help us tackle more polyhedrality problems from a different angle, and further investigation might be of independent interest.

\bibliographystyle{plain}
\bibliography{cite}

\begin{thebibliography}{10}

\bibitem{MR2676765}
Kent Andersen, Quentin Louveaux, and Robert Weismantel.
\newblock An analysis of mixed integer linear sets based on lattice point free
  convex sets.
\newblock {\em Mathematics of Operations Research}, 35(1):233--256, 2010.

\bibitem{MR2969261}
Gennadiy Averkov.
\newblock On finitely generated closures in the theory of cutting planes.
\newblock {\em Discrete Optim.}, 9(4):209--215, 2012.

\bibitem{bodur2017aggregation}
Merve Bodur, Alberto Del~Pia, Santanu~S Dey, Marco Molinaro, and Sebastian
  Pokutta.
\newblock Aggregation-based cutting-planes for packing and covering integer
  programs.
\newblock {\em Mathematical Programming}, pages 1--29, 2017.

\bibitem{MR313080}
V.~Chv\'{a}tal.
\newblock Edmonds polytopes and a hierarchy of combinatorial problems.
\newblock {\em Discrete Math.}, 4:305--337, 1973.

\bibitem{conforti2010polyhedral}
Michele Conforti, G{\'e}rard Cornu{\'e}jols, and Giacomo Zambelli.
\newblock Polyhedral approaches to mixed integer linear programming.
\newblock In {\em 50 years of integer programming 1958-2008}, pages 343--385.
  Springer, 2010.

\bibitem{Conforti:2014:IP:2765770}
Michele Conforti, Gerard Cornuejols, and Giacomo Zambelli.
\newblock {\em Integer Programming}.
\newblock Springer Publishing Company, Incorporated, 2014.

\bibitem{cook1990chvatal}
William Cook, Ravindran Kannan, and Alexander Schrijver.
\newblock Chv{\'a}tal closures for mixed integer programming problems.
\newblock {\em Mathematical Programming}, 47(1-3):155--174, 1990.

\bibitem{cornuejols2008valid}
G{\'e}rard Cornu{\'e}jols.
\newblock Valid inequalities for mixed integer linear programs.
\newblock {\em Mathematical Programming}, 112(1):3--44, 2008.

\bibitem{dash2020generalized}
Sanjeeb Dash, Oktay G{\"u}nl{\"u}k, and Dabeen Lee.
\newblock Generalized chv{\'a}tal-gomory closures for integer programs with
  bounds on variables.
\newblock {\em Mathematical Programming}, pages 1--33, 2020.

\bibitem{MR3670262}
Sanjeeb Dash, Oktay G\"{u}nl\"{u}k, and Diego~A. Mor\'{a}n~R.
\newblock On the polyhedrality of closures of multibranch split sets and other
  polyhedra with bounded max-facet-width.
\newblock {\em SIAM J. Optim.}, 27(3):1340--1361, 2017.

\bibitem{MR4207341}
Alberto Del~Pia, Dion Gijswijt, Jeff Linderoth, and Haoran Zhu.
\newblock Integer packing sets form a well-quasi-ordering.
\newblock {\em Oper. Res. Lett.}, 49(2):226--230, 2021.

\bibitem{dickson1913finiteness}
Leonard~Eugene Dickson.
\newblock Finiteness of the odd perfect and primitive abundant numbers with n
  distinct prime factors.
\newblock {\em American Journal of Mathematics}, 35(4):413--422, 1913.

\bibitem{goberna1998linear}
Miguel~A Goberna.
\newblock Linear semi-infinite optimization.
\newblock {\em Mathematical Methods in Practice 2}, 1998.

\bibitem{gruber1993handbook}
Peter~M Gruber and J{\"o}rg~M Wills.
\newblock {\em Handbook of convex geometry: Vol. A}.
\newblock North-Holland, 1993.

\bibitem{husseinov1999note}
F~H{\"u}sseinov.
\newblock A note on the closedness of the convex hull and its applications.
\newblock {\em Journal of Convex Analysis}, 6(2):387--393, 1999.

\bibitem{johnson2000progress}
Ellis~L Johnson, George~L Nemhauser, and Martin~WP Savelsbergh.
\newblock Progress in linear programming-based algorithms for integer
  programming: An exposition.
\newblock {\em Informs journal on computing}, 12(1):2--23, 2000.

\bibitem{klee1957extremal}
VL~Klee.
\newblock Extremal structure of convex sets.
\newblock {\em Archiv der Mathematik}, 8(3):234--240, 1957.

\bibitem{MR1922341}
Hugues Marchand, Alexander Martin, Robert Weismantel, and Laurence Wolsey.
\newblock Cutting planes in integer and mixed integer programming.
\newblock volume 123, pages 397--446. 2002.
\newblock Workshop on Discrete Optimization, DO'99 (Piscataway, NJ).

\bibitem{pashkovich2021aggregation}
Kanstantsin Pashkovich, Laurent Poirrier, and Haripriya Pulyassary.
\newblock The aggregation closure is polyhedral for packing and covering
  integer programs.
\newblock {\em Mathematical Programming}, pages 1--13, 2021.

\bibitem{richard2010group}
Jean-Philippe~P Richard and Santanu~S Dey.
\newblock The group-theoretic approach in mixed integer programming.
\newblock In {\em 50 Years of Integer Programming 1958-2008}, pages 727--801.
  Springer, 2010.

\bibitem{unpublishedconvexity}
Stephen~M. Robinson.
\newblock Convexity in finite-dimensional spaces.
\newblock unpublished, 2015.

\bibitem{schrijver1998theory}
Alexander Schrijver.
\newblock {\em Theory of linear and integer programming}.
\newblock John Wiley \& Sons, 1998.

\bibitem{studeny1993convex}
Milan Studen{\`y}.
\newblock Convex cones in finite-dimensional real vector spaces.
\newblock {\em Kybernetika}, 29(2):180--200, 1993.

\bibitem{MR1227751}
Milan Studen\'{y}.
\newblock Convex cones in finite-dimensional real vector spaces.
\newblock {\em Kybernetika (Prague)}, 29(2):180--200, 1993.

\bibitem{vielma2007constructive}
Juan~Pablo Vielma.
\newblock A constructive characterization of the split closure of a mixed
  integer linear program.
\newblock {\em Operations Research Letters}, 35(1):29--35, 2007.

\end{thebibliography}


\end{document}